\documentclass[12pt]{amsart}
\usepackage{amscd, amsfonts, amssymb}
\usepackage{fullpage}
\usepackage{latexsym}
\usepackage{verbatim}
\usepackage{enumerate}

\newcommand\inverse{{^{-1}}}

\newcommand{\FF}{{\mathbb F}}
\newcommand{\NN}{{\mathbb N}}
\newcommand{\ZZ}{{\mathbb Z}}

\DeclareMathOperator{\GL}{GL}

\DeclareMathOperator{\Hom}{Hom}

\DeclareMathOperator{\End}{End}

\numberwithin{equation}{section}

\newtheorem{lem}[equation]{Lemma}

\newtheorem{prop}[equation]{Proposition}

\theoremstyle{definition}

\newtheorem{exmp}[equation]{Example}

\newtheorem*{alg}{Algorithm}
\theoremstyle{remark}
\newtheorem{rem}[equation]{Remark}
\theoremstyle{remark}

\newtheorem{step}{Step}


\title[The Size of $\Hom(A,\GL_n(q))$.]
{Asymptotic Bounds for the Size of $\Hom(A,\GL_n(q))$.}

\author[M.\  Bate]{Michael Bate}
\address%[M.\  Bate]
{Department of Mathematics,
University of York,
York YO10 5DD.
United Kingdom}
\email{michael.bate@york.ac.uk}

\author[A. Gullon]{Alec Gullon}
\address%[A. Gullon]
{Department of Mathematics,
Lancaster University,
Lancaster,
LA1 4YF,
United Kingdom}
\email{a.gullon@lancaster.ac.uk}

\begin{document}

\begin{abstract}
Fix an arbitrary finite group $A$ of order $a$, and let $X(n,q)$ denote the set of homomorphisms
from $A$ to the finite general linear group $\GL_n(q)$.
The size of $X(n,q)$ is a polynomial in $q$.
In this note it is shown that generically this polynomial has degree $n^2(1-a^{-1}) - \epsilon_r$ and
leading coefficient $m_r$, where $\epsilon_r$ and $m_r$ are constants depending only on
$r := n \mod a$.
We also present an algorithm for explicitly determining these constants.
\end{abstract}

\maketitle

\section{Introduction}

Let $A$ be a finite group of order $a$ and let $X(n,q) = \Hom(A,\GL_n(q))$ denote the set of all homomorphisms
from $A$ to the general linear group of $n\times n$ invertible matrices with entries
in the finite field $\FF_q$.
Suppose that $\FF_q$ is a splitting field for $A$.
In \cite{ls}, Liebeck and Shalev provide upper and lower bounds for the size
of this set, which is a polynomial in $q$ \cite[Prop.\ 4.1]{chig}; see also \cite{bate}.
The bound presented in \cite[Theorem]{ls} has the following form:
$$
cq^{(n^2-r^2)(1-a^{-1})} \leq |X(n,q)| \leq dq^{n^2(1-a^{-1})},
$$
where $c$ is an absolute constant, $d$ is a constant depending only on $a$, and $r$ is the value of $n$ modulo $a$. Note that, as is pointed out in \cite{ls} and \cite{bate}, there is an absolute constant $\beta>0$ such that $\beta q^{n^2} \leq |\GL_n(q)| \leq q^{n^2}$ for all $n$ and $q$, so these bounds can be rewritten in terms of the order of $\GL_n(q)$.
The aim of this note is to show that there exists $N$ such that for all $n \geq N$
the leading term of the polynomial $|X(n,q)|$ has the form
$$
m_rq^{n^2(1-a^{-1}) - \epsilon_r},
$$
where (given a fixed group $A$) $m_r$ and $\epsilon_r$ are constants only depending on $r$,
and $N= a(a-1)$ will definitely suffice.
In particular, this leading coefficient and degree are independent of $q$.
We also present an algorithm for explicitly determining the values of $m_r$, $\epsilon_r$ and $N$ for any choice of
$A$.
The input needed for the algorithm is the degrees of the irreducible representations
for $A$ over a splitting field for $A$.

The paper is laid out as follows.
We begin in Section \ref{sec:prelims} by setting up some basic notation and recalling some of the
analysis from \cite{ls} and \cite{bate}, before moving on to the main results
in Section \ref{sec:theory}.
After giving some examples to illustrate various points of the paper in Section \ref{sec:egs},
in the final section of the paper we indicate how to relax some of the assumptions in force for the rest of the
paper, and also make some further remarks.

\section{Preliminaries}\label{sec:prelims}

Throughout, $A$ denotes a finite group with $a$ elements.
We use $q$ to denote the order of a finite field $\mathbb{F}_q$, so $q=p^d$ for some prime $p$ and positive integer $d$.
Our standing assumption on $q$ for most of this note (except in Section \ref{sec:extensions}) is that $\FF_q$ is a splitting field for $A$,
i.e., the characteristic $p$ of the field $\FF_q$ does not divide $a$ and all irreducible $\FF_qA$-modules
are absolutely irreducible.
By Schur's Lemma, given a simple $\FF_qA$-module $M$, we have $\End_{\FF_qA-{\rm mod}}(M) \simeq \FF_q$.

For an $n$-dimensional $\FF_q$-vector space $V$, we have the finite general linear group $\GL(V)$ which we freely identify
with $\GL_n(q)$, the group of invertible $n\times n$ matrices with entries in $\FF_q$, when it is convenient to do so.
We let $X(n,q)=\Hom(A,\GL_n(q))$ denote the set of homomorphisms from $A$ to $\GL_n(q)$ for each choice of $n$ and $q$,
and note that $\GL_n(q)$ acts on $X(n,q)$ by conjugation: given $\rho:A \to \GL_n(q)$ and $g \in \GL_n(q)$,
set $(g\cdot\rho)(a) = g\rho(a)g\inverse$ for all $a \in A$.
This breaks the set $X(n,q)$ up into $\GL_n(q)$-orbits, and one key part of the analysis in
\cite{ls} and \cite{bate} is to bound the size of each of these orbits.
This involves some basic representation theory, which we now recap.

Let $(M_1,\ldots,M_s)$ be a complete ordered tuple of pairwise non-isomorphic irreducible
(hence absolutely irreducible by our assumptions on the field)
$\FF_qA$-modules, and let $d_i = \dim(M_i)$ for each $i$.
Choose the labelling so that $M_1$ is the trivial module.
The degrees $d_i$ are the same for any splitting field for $A$, and $a = \sum_{i=1}^s d_i^2$.
Given any $\FF_qA$-module $V$, we have an isomorphism
$$
V \simeq n_1M_1 \oplus \cdots \oplus n_sM_s,
$$
where $n_iM_i$ denotes the direct sum of the module $M_i$ with itself $n_i$ times (we allow $n_i=0$ here).
For a given $\FF_qA$-module $V$, we therefore have an ordered $s$-tuple $(n_1,\ldots,n_s)$ of non-negative integers and
two $\FF_qA$-modules are isomorphic if and only if they have the same ordered $s$-tuple attached.
Moreover, if we restrict attention to $n$-dimensional modules for some fixed $n$, then the relevant $s$-tuples
$(n_1,\ldots,n_s)$
for which $\sum_{i=1}^s n_id_i = n$ also parametrise the $\GL_n(q)$-orbits in $X(n,q)$.
It follows from the analysis in \cite{ls} and \cite{bate} that, given a tuple $(n_1,\ldots,n_s)$, the
stabilizer associated to the corresponding orbit in $X(n,q)$ is isomorphic to a product $\prod_{i=1}^s \GL_{n_i}(q)$,
which allows us to write down the size of the orbit by the Orbit-Stabilizer Theorem.
The key to the approach presented in this note is to give a better estimate of the largest possible size for such an orbit,
and to show that such a size is attained, which improves on the upper and lower bounds presented in \cite{ls}.

\section{Results}\label{sec:theory}

Keep the notation from the previous section, and remember our standing assumption that $\FF_q$ is a splitting field for $A$.
Before stating and proving the main technical results needed for our algorithm, we introduce some more terminology.
Let $n \in \NN$.
We say that an ordered tuple $(n_1,\ldots,n_s)$ of integers (not necessarily non-negative) is
\emph{admissible for $n$} if $\sum_{i=1}^s n_id_i = n$; if the context is clear, then we simply say the tuple is admissible.
We call such an admissible tuple \emph{eligible} if, in addition, $n_i\geq 0$ for all $i$.
Finally, we call a tuple $(n_1,\ldots,n_s)$ which is admissible for $n \in \NN$ a \emph{minimal tuple for $n$} if
$\sum_{i=1}^s n_i^2$ is minimal amongst all admissible tuples for $n$.

\begin{lem}\label{lem:minimaltuple=>maximalorbit}
Fix $n \in \NN$. Then:
\begin{itemize}
\item[(i)] the parametrisation of the orbits in $X(n,q)$ by eligible tuples is independent of $q$;
\item[(ii)] for each eligible tuple $\mathbf{t} = (n_1,\ldots,n_s)$, there is a polynomial $f_\mathbf{t}(x) \in \ZZ[x]$ such that
the size of the corresponding orbit in $X(n,q)$ is $f_{\mathbf{t}}(q)$.
\item[(iii)] the polynomials of maximal degree from (ii) are precisely those corresponding to minimal tuples.
\end{itemize}
\end{lem}

\begin{proof}
Part (i) follows from the fact that the dimensions and number of isomorphism classes of absolutely irreducible $\FF_qA$-modules
over a splitting field are independent of that field.
Then (ii) follows because the shape of the stabilizer of a given orbit, as given at the end of the previous section, is independent of $q$.
Specifically, given any eligible tuple $(n_1,\ldots,n_s)$, the size of the associated orbit in $X(n,q)$ is $|\GL_n(q)|/\prod_{i=1}^s |\GL_{n_i}(q)|$,
which is a polynomial in $q$ \cite[Prop.\ 4.1]{chig} with leading term $q^{n^2-\sum_{i=1}^sn_i^2}$.
Therefore, the largest degree amongst all these polynomials is attained by those polynomials corresponding to tuples for which
$\sum_{i=1}^s n_i^2$ is minimal.
These are precisely the polynomials corresponding to minimal eligible tuples, which proves (iii).
\end{proof}

\begin{lem}\label{lem:finitenumberoftuples}
Given $n \in \NN$, there are finitely many minimal tuples for $n$.
\end{lem}

\begin{proof}
Since $d_1=1$, the tuple $(n,0,0,\ldots,0)$ is admissible for $n$, and hence for a minimal tuple $(n_1,\ldots,n_s)$ we have
$\sum_{i=1}^sn_i^2 \leq n^2$.
In particular, for each $i$ we have $-n\leq n_i \leq n$.
This gives rise to finitely many tuples, and all the minimal tuples lie amongst these.
\end{proof}

The above lemma shows that for each $n$ we have finitely many minimal tuples to worry about.
In fact, we can do much better than that, as the following results show.

\begin{lem}\label{lem:correspondence}
Given $n \in \NN$, write $n = ka+r$ where $k \in \NN\cup\{0\}$ and $0\leq r < a$.
\begin{itemize}
\item[(i)]  Suppose $(r_1,\ldots,r_s)$ is a minimal tuple for $r$.
Then $(kd_1+r_1,\ldots,kd_s+r_s)$ is a minimal tuple for $n$.
\item[(ii)] Suppose $(n_1,\ldots,n_s)$ is a minimal tuple for $n$.
Then $(n_1-kd_1,\ldots,n_s-kd_s)$ is a minimal tuple for $r$.
\end{itemize}
Hence minimal tuples for $n$ are in 1-1 correspondence with minimal tuples for $r$.
\end{lem}

\begin{proof}
First note that if $(r_1,\ldots,r_s)$ is an admissible tuple for $r$
then $\sum_{i=1}^s r_id_i = r$, so
$$
\sum_{i=1}^s (kd_i + r_i)d_i = k\sum_{i=1}^s d_i^2 + \sum_{i=1}^sr_id_i = ka+r = n,
$$
and hence $(kd_1+r_1,\ldots,kd_s+r_s)$ is admissible for $n$.
Conversely, if $(n_1,\ldots,n_r)$ is admissible for $n$, then
$$
\sum_{i=1}^s (n_i-kd_i)d_i = \sum_{i=1}^sn_id_i - k\sum_{i=1}^s d_i^2 =n- ka=r,
$$
so $(n_1-kd_1,\ldots,n_s-kd_s)$ is admissible for $r$.

Now suppose $(r_1,\ldots,r_s)$ is minimal for $r$ and $(n_1,\ldots,n_s)$ is minimal for $n$.
Since $(r_1,\ldots,r_s)$ is minimal for $r$ and $(n_1-kd_1,\ldots,n_s-kd_s)$ is admissible for $r$, we have
\begin{align*}
\sum_{i=1}^s(kd_i+r_i)^2 &=k^2\sum_{i=1}^s d_i^2 + 2k\sum_{i=1}^s r_id_i +\sum_{i=1}^s r_i^2\\
                   &= k^2a+2kr+ \sum_{i=1}^sr_i^2\\
                   &\leq k^2a+2kr+\sum_{i=1}^s (n_i-kd_i)^2\\
                   &= k^2a+2kr+\sum_{i=1}^s n_i^2 - 2k\sum_{i=1}^s n_id_i + k^2\sum_{i=1}^sd_i^2\\
                   &= \sum_{i=1}^s n_i^2 + 2k^2a+2kr- 2kn \\
                   &= \sum_{i=1}^s n_i^2.
\end{align*}
But $(n_1,\ldots,n_r)$ is minimal for $n$, so we must actually have equality here and hence $(kd_1+r_1,\ldots,kd_s+r_s)$
is also minimal for $n$.
Using this equality, we now also have
\begin{align*}
\sum_{i=1}^s(n_i-kd_i)^2 &=\sum_{i=1}^s n_i^2 - 2k\sum_{i=1}^s n_id_i +k^2\sum_{i=1}^s d_i^2\\
                   &= \sum_{i=1}^s (kd_i+r_i)^2 - 2kn+k^2a\\
                   &= 2k^2a +2kr + \sum_{i=1}^sr_i^2 -2k(ka+r) \\
                   &= \sum_{i=1}^s r_i^2,
\end{align*}
so $(n_1-kd_1,\ldots,n_s-kd_s)$ is a minimal tuple for $r$.
\end{proof}

\begin{rem}
Lemma \ref{lem:correspondence} is really at the heart of this note.
It shows that, despite the fact that the whole set $X(n,q)$ gets more and more complicated as $n$ grows, we
can still exert some control over the orbits which are largest in the sense of Lemma \ref{lem:minimaltuple=>maximalorbit}(iii).
One cannot hope for this to be true for smaller orbits, because as $n$ grows, so does the number of eligible tuples
and hence the total number of orbits.
\end{rem}

\begin{lem}\label{lem:OKafterN}
There exists $N \in \NN$ such that for all $n\geq N$, all minimal tuples for $n$ are eligible and the number
of minimal tuples only depends on the value of $n$ modulo $a$.
\end{lem}

\begin{proof}
For any $n \in \NN$, the number of minimal tuples for $n$ is the same as the number of minimal tuples for $r$, where $n=ka+r$ with $k \in \NN\cup\{0\}$ and $0\leq r<a$,
by Lemma \ref{lem:correspondence}, so this number only depends on the value of $n$ modulo $a$.
Moreover, the minimal tuples for $n$ all have the form $(kd_1+r_1,\ldots,kd_s+r_s)$ where $(r_1,\ldots,r_s)$ is a minimal
tuple for $r$, again by Lemma \ref{lem:correspondence}.
There are finitely many values $r_i$ as $r$ runs over all integers between $0$ and $a-1$ by the argument in the proof of Lemma \ref{lem:finitenumberoftuples},
and we just need to choose $N$ large enough so that for all $n \geq N$, all possible values $kd_i + r_i \geq 0$, which can clearly be done.
After this point, all the minimal tuples are also eligible.
\end{proof}

\begin{rem}
By the argument in the proof of Lemma \ref{lem:finitenumberoftuples}, if $(r_1,\ldots,r_s)$ is a minimal tuple for $0\leq r <a$,
the minimal possible value for any $r_i$ is $-r$.
Since each $d_i\geq 1$ this means that $kd_i+r_i\geq 0$ as long as $k>a-1$ for any choice of $0\leq r <a$,
so we could choose $N = a(a-1)$ in Lemma \ref{lem:OKafterN} if we wanted a concrete bound.
However, in practice, as we shall see, the best value for $N$ is often much less than this.
\end{rem}

\begin{rem}
When $n = ka$ is a multiple of $a$, the tuple $(kd_1,\ldots,kd_r)$ is the unique minimal eligible tuple for $n$.
This is because this tuple gives a global minimum for the value $\sum_{i=1}^s x_i^2$ amongst all tuples of real numbers $(x_1,\ldots,x_s)$ satisfying
the constraint $\sum_{i=1}^s x_id_i = n$, as can be verified using some basic calculus.
\end{rem}

\begin{prop}\label{prop:mainstatement}
Suppose $0\leq r <a$.
Let $m_r$ be the number of minimal tuples for $r$, and let $(r_1,\ldots,r_s)$ be one of the minimal tuples for $r$.
Let $\epsilon_r = \sum_{i=1}^sr_i^2 - r^2a^{-1}$.
Given $n \in \NN$ with $n \geq N$, where $N$ is as in Lemma \ref{lem:OKafterN}, write $n = ka+r$, where $k \in \ZZ$ and $0\leq r <a$.
Then
\begin{itemize}
\item[(i)] $\epsilon_r \geq 0$ (with equality if and only if $r=0$);
\item[(ii)] there exists a polynomial $f_n(x) \in \ZZ[x]$, independent of $q$, whose leading term is $m_rx^{n^2(1-a^{-1})-\epsilon_r}$
such that $|X(n,q)| = f_n(q)$.
\end{itemize}
\end{prop}

\begin{proof}
(i). The global minimum value for $\sum_{i=1}^s x_i^2$
amongst all real $s$-tuples $(x_1,\ldots,x_s)$ satisfying $\sum_{i=1}^s x_id_i = r$ is given by the tuple $(a^{-1}rd_1,\ldots,a^{-1}rd_s)$,
and hence $\sum_{i=1}^s r_i^2 \geq \sum_{i=1}^s a^{-2}r^2d_i^2 = a^{-1}r^2$, as required.
It is clear that if $r=0$, then $\epsilon_r=0$.
For the converse note that we get $\epsilon_r=0$ if and only if the global minimum tuple is an integer tuple,
that is if and only if $a^{-1}rd_i \in \ZZ$ for all $i$.
But since $d_1 = 1$ (the degree of the trivial irreducible representation), this means that $a^{-1}r \in \ZZ$.
Since $r<a$, this is only possible if $r=0$.

(ii). By the definition of $N$ from Lemma \ref{lem:OKafterN}, all minimal tuples for $n$ are eligible and there are precisely $m_r$ of them.
Moreover, each minimal tuple for $n$ has the form $(kd_1+r_1',\ldots,kd_s+r_s')$, where $(r_1',\ldots,r_s')$ is a minimal tuple for $r$.
Now note that for any minimal tuple $(r_1',\ldots,r_s')$ for $r$ we have $\sum_{i=1}^s (r_i')^2 = \sum_{i=1}^s r_i^2$,
where $(r_1,\ldots,r_s)$ is the fixed minimal tuple picked in the statement of the result.
Then
\begin{align*}
\sum_{i=1}^s (kd_i+r_i')^2 &= k^2\sum_{i=1}^s d_i^2 + 2k\sum_{i=1}^sr_i'd_i + \sum_{i=1}^s (r_i')^2\\
                           &= k^2a + 2kr + \sum_{i=1}^s r_i^2 \\
                           &= a^{-1}(k^2a^2 + 2akr + r^2) -r^2a^{-1} + \sum_{i=1}^s r_i^2 \\
                           &= a^{-1}n^2 + \epsilon_r.
\end{align*}
By Lemma \ref{lem:minimaltuple=>maximalorbit}(iii), the minimal tuples for $n$ give rise to the orbits whose orders are polynomials of maximal degree
amongst the orders of all orbits in $X(n,q)$, and the order of each of these orbits is a polynomial in $q$ with leading term $q^{n^2- a^{-1}n^2 - \epsilon_r}$.
Since $X(n,q)$ is the disjoint union of all of the orbits it contains, the order of $X(n,q)$ is a polynomial in $q$ with leading term $m_rq^{n^2(1-a^{-1})-\epsilon_r}$.
None of the arguments used to derive this result rely on the actual value of $q$, only that $\FF_q$ is a splitting field for $A$.
Since the degrees $d_i$ are all the same for any splitting field, we get the result.
\end{proof}

We summarise the results obtained in the form of an algorithm:

\begin{alg} The following steps will allow one to find the numbers $m_r$, $\epsilon_r$ and $N$
from Proposition \ref{prop:mainstatement}, and hence calculate the highest degree term of the polynomial
$|X(n,q)|$ for any $n \geq N$.

\begin{step}
For each $0  \leq r  < a$ find all minimal tuples of integers for $r$; that is, find all tuples $(r_1,\ldots,r_s)$
satisfying $\sum_{i=1}^sr_id_i = r$ and minimising the value of $\sum_{i=1}^s r_i^2$.
For each $r$, record the number $m_r$ of minimal tuples found, and the number $\epsilon_r := \sum_{i=1}^s r_i^2 - r^2a^{-1}$,
where $(r_1,\ldots,r_s)$ is one of the minimal tuples.
\end{step}

\begin{step}
Find the smallest $b \in \NN \cup \{0\}$ such that $bd_i + r_i \geq 0$ for all $r_i$ from step $1$. Then set $N = ba$.
\end{step}
\end{alg}

\section{Examples}\label{sec:egs}

We now present some examples of our algorithm and its results when applied to some groups which are
relatively easy to handle.
For a given minimal tuple $(r_1,\ldots,r_s)$, we denote $\sum_{i=1}^s r_i^2$ by $S_r$.

\subsection{Abelian groups}
If $A$ is abelian and $\FF_q$ is a splitting field for $A$, then there are $a$ distinct classes of irreducible
representations of $A$ and they are all one-dimensional.
Therefore, for $0\leq r < a$, a minimal tuple is just found by filling $r$ spaces with a $1$
and $a-r$ with a zero.
This means that all minimal tuples are eligible, so $N=0$, $m_r$ is the binomial coefficient
$\binom{a}{r}$, and $S_r = r$.
Therefore $\epsilon_r = r-r^2a^{-1}$ and the leading term of the polynomial $|X(n,q)|$ is
$\binom{a}{r}q^{n^2(1-a^{-1}) - r +r^2a^{-1}}$.

\subsection{Dihedral groups $D_m$}
Let $A = D_m$ be the dihedral group of order $a=2m$.
For $\FF_q$ to be a splitting field for $A$, it is enough that $\FF_q$ contains all elements of the form $\zeta + \zeta^{-1}$, where $\zeta$ is a root of the $m^{\rm th}$ cyclotomic polynomial.
Over a splitting field, if $m$ is odd, then $A$ has two irreducible representations of degree $1$ and $\frac{m-1}{2}$ of degree $2$,
and if $m$ is even we get four irreducible representations of degree $1$ and $\frac{m-2}{2}$ of degree $2$.
Hence, we split into two cases:

\smallskip
\noindent
\textbf{$m = 2l+1$ is odd.}
In this case for any even $r = 2k$ with $0\leq r <2n$, we choose $k$ of the $m$ representations of degree $2$ to come up with a minimal tuple;
there are $\binom{m}{k}$ ways to do this, so that is the value of $m_r$, and $S_r = k = \frac{r}{2}$.
For an odd $r = 2k+1$, we again choose $k$ of the degree $2$ representations, and then one of degree $1$;
there are therefore $m_r=2\binom{m}{k}$ ways to do this, each one with $S_r = k+1 = \frac{r+1}{2}$.
Hence we have that the leading term of the polynomial $|X(n,q)|$ is
$\binom{m}{k}q^{n^2(1-a^{-1}) - \frac{r}{2} +r^2a^{-1}}$ if $r=2k$ is even
and $2\binom{m}{k}q^{n^2(1-a^{-1}) - \frac{r+1}{2} +r^2a^{-1}}$ if $r=2k+1$ is odd.

\smallskip
\noindent
\textbf{$m = 2l$ is even.}
In this case for any even $r = 2k$ with $0\leq r <2n$, we choose $k$ of the $m$ representations of degree $2$ to come up with a minimal tuple;
there are $\binom{m}{k}$ ways to do this, so that is the value of $m_r$, and $S_r = k = \frac{r}{2}$.
For an odd $r = 2k+1$, we again choose $k$ of the degree $2$ representations, and then one of degree $1$;
there are therefore $m_r=4\binom{m}{k}$ ways to do this, each one with $S_r = k+1 = \frac{r+1}{2}$.
Hence we have that the leading term of the polynomial $|X(n,q)|$ is
$\binom{m}{k}q^{n^2(1-a^{-1}) - \frac{r}{2} +r^2a^{-1}}$ if $r=2k$ is even
and $4\binom{m}{k}q^{n^2(1-a^{-1}) - \frac{r+1}{2} +r^2a^{-1}}$ if $r=2k+1$ is odd.

\subsection{The symmetric group $\mathfrak{S}_4$}
If $A=\mathfrak{S}_4$, which has order $a=24$, then
any field of characteristic not $2$ or $3$ is a splitting field for $A$, and
the degrees of the irreducible representations over such a field are $1$, $1$, $2$, $3$ and $3$.
According to our algorithm we need to determine the minimal tuples for all $0 \leq r < 24$.
The relevant data is summarised in the following table.

\renewcommand\arraystretch{1.4}
\smallskip
\begin{center}
	\begin{tabular}{| c | c  | c | c | c || c | c | c | c | c |} \hline
		$r$ 	& $m_r$		& \text{sample tuple}	& $S_r$ & $\epsilon_r$	     & $r$	& $m_r$	& \text{sample tuple }	& $S_r$	& $\epsilon_r$	\\ \hline
		$0$ 	& $1$		& $(0,0,0,0,0)$		 	& $0$	& $0$			     & $12$	& $4$	& $(1,0,1,2,1)$			& $7$	& $1$			\\
		$1$		& $2$		& $(1,0,0,0,0)$			& $1$	& $\frac{23}{24}$	 & $13$	& $2$	& $(1,1,1,2,1)$			& $8$	& $\frac{23}{24}$			\\
		$2$		& $1$		& $(0,1,0,0,0)$			& $1$	& $\frac{5}{6}$		 & $14$	& $1$	& $(0,0,1,2,2)$			& $9$	& $\frac{5}{6}$			\\
		$3$		& $2$		& $(0,0,0,1,0)$			& $1$	& $\frac{5}{8}$		 & $15$	& $2$	& $(1,0,1,2,2)$			& $10$	& $\frac{5}{8}$			\\
		$4$		& $4$		& $(1,0,0,1,0)$			& $2$	& $\frac{4}{3}$		 & $16$	& $1$	& $(1,1,1,2,2)$			& $11$	& $\frac{1}{3}$			\\
		$5$		& $2$		& $(0,0,1,1,0)$			& $2$	& $\frac{23}{24}$	 & $17$	& $2$	& $(1,0,2,2,2)$			& $13$	& $\frac{23}{24}$			\\
		$6$		& $1$		& $(0,0,0,1,1)$			& $2$	& $\frac{1}{2}$		 & $18$	& $1$	& $(1,1,2,2,2)$			& $14$	& $\frac{1}{2}$			\\
		$7$		& $2$		& $(1,0,0,1,1)$			& $3$	& $\frac{23}{24}$	 & $19$	& $2$	& $(1,1,1,3,2)$			& $16$	& $\frac{23}{24}$			\\
		$8$		& $1$		& $(0,0,1,1,1)$			& $3$	& $\frac{1}{3}$		 & $20$	& $4$	& $(1,0,2,3,2)$			& $18$	& $\frac{4}{3}$			\\
		$9$		& $2$		& $(1,0,1,1,1)$			& $4$	& $\frac{5}{8}$		 & $21$	& $2$	& $(1,1,2,3,2)$			& $19$	& $\frac{5}{8}$			\\
		$10$	& $1$		& $(1,1,1,1,1)$			& $5$	& $\frac{5}{6}$		 & $22$	& $1$	& $(1,1,1,3,3)$			& $21$	& $\frac{5}{6}$			\\
		$11$	& $2$		& $(0,0,1,2,1)$			& $6$	& $\frac{23}{24}$	 & $23$	& $2$	& $(1,0,2,3,3)$			& $23$	& $\frac{23}{24}$			\\ \hline
	\end{tabular}
\end{center}
\smallskip

Of note here is the fact that in this case for every $r$ all the minimal tuples are eligible tuples,
and hence for this example the value of $N=0$ (so our result is valid for all $n$).
It is relatively straightforward to show that this is a general phenomenon which occurs when the degrees of the irreducible representations for $A$ can be put into an ordered list $1 = d_1 \leq d_2 \leq \cdots \leq d_s$ with $d_i - d_{i-1} \leq 1$ for $1 < i \leq s$.

It is also worth noting that this small example already shows that
the value of the ``error term'' $\epsilon_r$ can be greater than $1$, so finding the degree of the polynomial
$|X(n,q)|$ is more complicated than simply taking the integer part of $n^2(1-a^{-1})$.

\subsection{Further Examples}
Let $A=\mathfrak{S}_5$, so $a=120$. Any field of characteristic larger than $5$ is a splitting field for $\mathfrak{S}_5$, and
the degrees of the irreducible representations over such a field are $1,1,4,4,5,5$ and $6$.
When $r=3$, the tuple $(0,-1,1,0,0,0,0)$ is one of four minimal tuples (the others being those naturally obtained from this one
by permuting amongst entries of equal degree).
Hence we cannot take $N=0$ as we have negative entries for at least one value of $r$.
In this case, the value of $N$ provided by our algorithm is $N = a = 120$.
For similar reasons, when $A=\mathfrak{S}_6$, we also need to go up to $N=a=720$.
We have also calculated directly that value of $N$ for all groups of order $a \leq 80$ is either $0$ or $a$.

\section{Extensions and Further Remarks}\label{sec:extensions}

In this section we outline various ways to extend the work presented, either by relaxing some of the standing assumptions made in Section \ref{sec:prelims} or by changing the groups involved.
We also point out an application of this work to the study of representation varieties.
We begin by discussing the restrictions we have placed on the field $\FF_q$.

\subsection{The assumption that $\FF_q$ is a splitting field}
We have had the standing assumption that $\FF_q$ is a splitting field for $A$.
This means, in particular, that the characteristic $p$ of $\FF_q$
is coprime to the size $a$ of the group $A$ and that every irreducible $\FF_qA$-module is absolutely irreducible.
This assumption allows us to assume that all modules encountered are semisimple and that the irreducible summands encountered do not really depend on the field in any essential way.

Suppose we relax the assumption that all irreducible modules are absolutely irreducible, but retain for now the assumption that $q$ and $a$ are coprime
(this is the situation in \cite{ls}, for example).
Then we can still write down a basic set $(N_1,\ldots,N_t)$, say, of irreducible $\FF_qA$-modules.
It follows from Schur's Lemma that $\End_{\FF_qA-{\rm mod}}(N_i)$ is a division ring for each $i$,
and since any finite division ring is a field it is not hard to
see that we have $\End_{\FF_qA-{\rm mod}}(N_i) \simeq \FF_{q^{e_i}}$ for some $e_i\geq 1$.
Moreover, if we extend scalars to $\FF_{q^{e_i}}$,
then the module $N_i$ splits into a direct sum of $e_i$ absolutely
irreducible $\FF_{q^{e_i}}A$-modules, $(M_{i1},\ldots,M_{ie_i})$ say, which form a single orbit under the action of the Galois group
${\rm Gal}(\FF_{q^{e_i}}/\FF_{q}) \simeq \ZZ_{e_i}$.
Denote the dimension of each $M_{ij}$ over $\FF_{q^{e_i}}$ by $d_{ij}$, and note that $d_{ij} = d_{i1}$ for all $1\leq j \leq e_i$.
Note also that the $\FF_q$-dimension of $N_i$ must be $d_{i1}e_i$.
Conversely, given an absolutely irreducible $\FF_{q^e}A$-module $M$ over some extension $\FF_{q^e}$ of $\FF_q$,
taking the direct sum of the distinct ${\rm Gal}(\FF_{q^e}/ \FF_q)$-
conjugates of $M$ forms a module which arises from precisely one of the
$N_i$ by extension of scalars from $\FF_q$ to $\FF_{q^e}$.
For justification of the claims above, see results in \cite[Sec.\ 7]{curtisreiner}, in particular Cor. 7.11 and Prop. 7.18.

In this way, one can retrieve all of the information necessary to mimic the proofs and constructions in
Section \ref{sec:theory} over $\FF_q$.
In particular, the degrees $d_{ij}$ occurring above (with their multiplicities $e_i$) are precisely the degrees of the distinct
representatives of the isomorphism classes of absolutely irreducible $KA$-modules over some sufficiently large extension
$K$ of $\FF_q$ (it suffices to extend to a finite field containing all the $\FF_{q^{e_i}}$).
We therefore have that $\sum_{i=1}^t\sum_{j=1}^{e_i} d_{ij}^2 = a$.
Within this set-up, one can still calculate stabilisers of representations -- we see direct products of $\GL_n(q^{e_i})$s --
and one can analyse $t$-tuples of integers $(m_1,\ldots,m_t)$ such that $\sum_{i=1}^t m_id_{i1}e_i = n$.
This amounts to the same thing as analysing those $s$-tuples of integers $(n_{11},\ldots,n_{te_t})$ such that $\sum_{i=1}^s\sum_{j=1}^{e_i} n_{ij}d_{ij} = n$
subject to the additional constraint that $n_{i1} = n_{ij}$ for all $1\leq j \leq e_i$.
Presenting it in this way shows that, given the degrees $d_{ij}$, we actually need to identify a subset of the eligible $s$-tuples for the number $n$
from Section \ref{sec:theory} -- those tuples satisfying the given additional constraint.
The analysis in Section \ref{sec:theory} now proceeds almost unchanged: certain of these tuples will be minimal, and after a certain point, all minimal
tuples will be eligible.
After this point, the leading coefficient and power of $q$ in the polynomial $|X(n,q)|$ will only depend
on the value of $n$ modulo $a$.
Furthermore, when $n$ is divisible by $a$, the unique minimal tuple described in the proof of Proposition \ref{prop:mainstatement}(i) actually
satisfies the additional constraint given, so we get a leading term $q^{n^2(1-a^{-1})}$ in this case.

The preceding remarks are perhaps made more transparent with a couple of simple examples.
We maintain notation from previous sections.
\begin{exmp}
Let $d$ be an odd positive integer, let $q=2^d$, and let $A$ be the cyclic group of order $3$, generated by the element $x$ of order $3$, say.
We need three distinct cube roots of $1$ to realise all the absolutely irreducible representations of $A$;
since $d$ is odd, $\FF_q$ does not contain three distinct cube roots of $1$, but $\FF_{q^2}$ does.
Over $\FF_q$, $A$ has the trivial representation $N_1$ and another irreducible representation $N_2$ of degree $2$;
$N_2$ can be realised concretely by sending $x$ to the matrix $\left(\begin{array}{cc} 1&1\\1&0 \end{array}\right)$
(note that this matrix does have order $3$ when the characteristic is $2$, and it has characteristic polynomial $X^2+X+1$, so is not diagonalisable
unless the field contains a nontrivial cube root of $1$).
If we extend scalars by adjoining a root of $X^2+X+1$ (i.e., move to the field $\FF_{q^2}$), then $N_2$ splits into two one-dimensional
modules.

Now, we have only have one representation of $A$ over $\FF_q$ of degree $1$ -- the trivial representation --
so we have in this case that that $|X(1,q)| = 1$.
On the other hand, if we extend scalars to $\FF_{q^2}$, we get $|X(1,q^2)| = 3$.
In fact, for any $n \in \NN$, we have
$$
|X(3n+1,q)| = q^{\frac{2}{3}(3n+1)^2 - \frac{2}{3}} + \textrm{ lower order terms},
$$
whereas
$$
|X(3n+1,q^2)| = 3(q^2)^{\frac{2}{3}(3n+1)^2 - \frac{2}{3}} + \textrm{ lower order terms}.
$$
This shows how the leading coefficient of $|X(n,q)|$ can vary with $q$ when we work with fields which are not necessarily splitting fields.
\end{exmp}

\begin{exmp}
Let $q=3$ and let $A$ be the dihedral group of order $10$.
Over $\FF_3$ there are three irreducible representations of $A$: two of dimension $1$, $N_1$ and $N_2$, say, and a single $4$-dimensional representation, say $N_3$.
If we extend scalars to $\FF_9$, then $N_4$ splits into the two familiar two-dimensional representations
(those we see ``generically'' by considering dihedral groups as groups of plane rotations and reflections);
call these two-dimensional modules $M_3$ and $M_3'$, and denote the
modules given by $N_1$ and $N_2$ after extension of scalars by $M_1$ and $M_2$.

Over $\FF_3$ there are four ways to build a representation of degree $3$ -- $N_1\oplus N_1 \oplus N_1$ or $N_1\oplus N_1 \oplus N_2$ or $N_1\oplus N_2 \oplus N_2$ or $N_2\oplus N_2 \oplus N_2$
-- and amongst these the middle two correspond to minimal tuples.
For these, the value of our error term $\epsilon_r = 1^2 + 2^2 - \frac{9}{10} = \frac{21}{10}$.
On the other hand, if we work over $\FF_9$, then we can build a $3$-dimensional representation $M_1\oplus M_3$, for example, for which the error term is $1^2 + 1^2 - \frac{9}{10} = \frac{11}{10}$.
Hence, whenever $n$ is congruent to $3$ modulo $10$,
$$
\deg (|X(n,3)|) = \frac{9}{10}n^2 - \frac{21}{10},
$$
whereas
$$
\deg(|X(n,9)|) = \frac{9}{10}n^2 - \frac{11}{10}.
$$
This shows how the degree of the polynomial $|X(n,q)|$ can vary with $q$ when we work with fields which are not necessarily splitting fields.
\end{exmp}

Progress is even possible if we drop the assumption that $q$ and $a$ are coprime, so that not all $\FF_qA$-modules are semisimple.
Some sort of semisimplicity assumption is necessary, as is shown by \cite[Example 1]{ls},
but it is possible to make progress in the modular case if one is willing to replace the set
$\Hom(A,\GL_n(q))$ with the set $\Hom_{cr}(A,\GL_n(q))$ of all homomorphisms from $A$ to $\GL_n(q)$ such that the associated representation is semisimple
(completely reducible);
this is the standing assumption in the paper \cite{bate}, for example.
If one does this, then similar results are possible to those in Section \ref{sec:theory},
but one has to work a bit harder.
For example, instead of results depending on $a$,
one has to use the dimension of the socle of the group algebra $\FF_qA$ (denote this dimension by $b$),
and one cannot hope for the results obtained to be independent of $q$.
However, one should still expect that the leading term of the
polynomial $|\Hom_{cr}(A,\GL_n(q))|$ will have the form $m_rq^{n^2(1-b^{-1})-\epsilon_r}$,
where $r$ is the value of $n$ modulo $b$, and $m_r$ and $\epsilon_r$ can be determined
by procedures similar to those laid out in Secion \ref{sec:theory}.

\subsection{Changing the target group}
The main point of the paper \cite{bate} is to produce bounds similar to those in \cite{ls},
replacing $\GL_n(q)$ with a unitary, symplectic or orthogonal group,
see \cite[Thm.\ B, Thm.\ C, Thm.\ D]{bate}.
This is achieved at the expense of knowing a bit more information about the simple modules for the group $A$; for
example, one needs to know how many of the simple modules are self-dual.
However, armed with this knowledge, an approach similar to that given in this note would produce
similar results for these cases too.

\subsection{Dimensions of representation varieties}
Let $K$ be an algebraically closed field, and let $G=\GL_n(K)$.
The set $X:=\Hom(A,G)$ of homomorpisms from $A$ to $G$
is an example of a \emph{representation variety} \cite{lubotzkymagid}
(it can be realised as a closed subvariety of the $a$-fold cartesian product $G^a$).
The linear algebraic group $G$ acts on $X$ by restriction of the simultaneous conjugation
action on $G^a$ and, as is observed in \cite[Sec.\ 2]{ls},
under the assumption that $K$ has characteristic zero or coprime to $a$,
the $G$-orbits in $X$ are the irreducible components of $X$.
The dimension of such an orbit is the dimension of $G$ minus the dimension of the associated stabilizer.
The analysis in this paper shows that the maximal dimension arising is precisely $n^2(1-a^{-1}) - \epsilon_r$,
where the notation is that in Proposition \ref{prop:mainstatement}, and this is therefore the dimension of $X$.
Moreover, the number $m_r$ is precisely the number of irreducible components of maximal dimension in $X$.

%%%%%%%%%%%%%%%%%%%%%%%%%%%%%%%%%%%%%%%%%%%%%%%%%%%%%%%%%%%%%%%%%%%%%%
%%%%%%%%%%%%% bibliography
%%%%%%%%%%%%%%%%%%%%%%%%%%%%%%%%%%%%%%%%%%%%%%%%%%%%%%%%%%%%%%%%%%%%%%

\end{document}